\documentclass[reqno]{llncs}
\usepackage[utf8]{inputenc}
\usepackage{amsmath}
\usepackage{amsfonts}
\usepackage{amssymb}
\usepackage{enumerate}
\usepackage{subcaption}
\usepackage{algorithm}
\usepackage{algpseudocode}
\usepackage{tikz}
\usepackage{here}
\usepackage{centernot}

\usepackage{array}
\newcolumntype{C}{ >{\centering\arraybackslash} m }

\usepackage[left=1.5in, right=1.5in, top=.9in, bottom=.9in]{geometry}
\usepackage[colorlinks=true,breaklinks=true,bookmarks=true,urlcolor=blue,
     citecolor=blue,linkcolor=blue,bookmarksopen=false,draft=false]{hyperref}

\def\URL#1{\href{#1}{#1}}         
    
\usepackage{cleveref}

\newcommand{\R}{\mathbb R}

\newcommand\veb{{\ve b}}
\newcommand\vecc{{\ve c}}
\newcommand\ved{{\ve d}}

\newcommand\veg{{\ve g}}

\newcommand\vex{{\ve x}}
\newcommand\vey{{\ve y}}

\newcommand{\eoproof}{\hspace*{\fill} $\square$ \vspace{5pt}}

\def\ve#1{\mathchoice{\mbox{\boldmath$\displaystyle\bf#1$}}
{\mbox{\boldmath$\textstyle\bf#1$}}
{\mbox{\boldmath$\scriptstyle\bf#1$}}
{\mbox{\boldmath$\scriptscriptstyle\bf#1$}}}

\begin{document}

\title{An implementation of steepest-descent augmentation for linear programs}

\author{Steffen Borgwardt\inst{1}\and Charles Viss\inst{2}}


\institute{\email{\href{mailto:steffen.borgwardt@ucdenver.edu}{steffen.borgwardt@ucdenver.edu}};
University of Colorado Denver \and
\email{\href{mailto:charles.viss@ucdenver.edu}{charles.viss@ucdenver.edu}};
University of Colorado Denver 
}

\date{\today}

\maketitle

\begin{abstract}
Generalizing the simplex method, circuit augmentation schemes for linear programs follow circuit directions through the interior of the underlying polyhedron. Steepest-descent augmentation is especially promising, but an implementation of the iterative scheme is a significant challenge. We work towards a viable implementation through a model in which a single linear program is updated dynamically to remain in memory throughout. Computational experiments exhibit dramatic improvements over a naïve approach and reveal insight into the next steps required for large-scale computations.
\end{abstract}


\section{Introduction}

We consider the optimization of a linear objective $\vecc \in \R^n$ over a general polyhedron of the form
\begin{align*}
    P = \{ \vex \in \R^n \colon A \vex = \veb, \  B \vex \leq \ved \},
\end{align*}
where $A \in \R^{m_A \times n}$ and $B \in \R^{m_B \times n}$. Recall from \cite{bfh-14,bfh-16} the definition of the set of \textit{circuits} $\mathcal{C}(A, B)$ of $P$:

\begin{definition}[Circuits]\label{def:circuits}
For a polyhedron $P = \{ \vex \in \R^n \colon A \vex = \veb, B \vex \leq \ved \}$, \textit{the set of circuits of $P$}, denoted $\mathcal{C}(A,B)$, consists of those $\veg \in \ker(A) \setminus \{ \ve 0 \}$ normalized to coprime integer components for which $B \veg$ is support-minimal over the set $\{B \vex \colon \vex \in \ker(A) \setminus \{\ve0 \} \}$.
\end{definition}

Circuits first appeared in the literature as the \textit{elementary vectors} of a subspace \cite{r-69}. Geometrically, the set of circuits consists of all potential edge directions of $P$ as the right-hand side vectors $\veb$ and $\ved$ vary \cite{g-75}. For general purposes, any normalization which results in a unique positive and negative representative for each of these one-dimensional directions can be used when working with circuits. We refer to such a scalar multiple of a circuit $\veg \in \mathcal{C}(A,B)$ as a \textit{circuit direction} of $P$. The integer normalization in \Cref{def:circuits} allows for intuitive interpretations of the circuits of many polyhedra from combinatorial optimization \cite{b-13,bdfm-18,bh-17,bv-17,kps-17}.

Note that $\mathcal{C}(A, B)$ is dependent on the representation of a polyhedron. In fact, converting between representations can introduce exponentially many additional directions to the set of circuits \cite{bv-18}. In this paper, we therefore use the above representation for $P$ which generalizes both standard form and canonical form so that no conversion between representations is required.

As a generalization of the set of edge directions of $P$, the set of circuits is a \textit{universal test} set for any linear program over the polyhedron \cite{g-75}; i.e., given a feasible solution $\vex_0$ to the linear program $\min \{\vecc^T \vex \colon \vex \in P\}$, either $\vex_0$ is an optimal solution or there exists a circuit $\veg \in \mathcal{C}(A, B)$ and a step size $\alpha > 0$ such that $\vex_0 + \alpha \veg \in P$ and $\vecc^T \veg < 0$. Circuits are therefore used in the development of \textit{augmentation schemes} for solving linear programs in which successive, improving, maximal steps are taken along circuit directions until an optimal solution is reached or the problem is determined to be unbounded \cite{dhl-15,hor-13,how-11}. These schemes generalize the simplex method in that they follow circuits, the potential edge directions of the underlying polyhedron; however, their steps are not restricted to only the actual edges of the polyhedron and may traverse its interior. Hence, the \textit{polynomial Hirsch conjecture} -- which states that the \textit{combinatorial diameter} of a polyhedron can be bounded polynomially -- need not be true in order for there to exist a strongly-polynomial time circuit augmentation scheme for linear programming. For this reason, there has been recent interest in the study of the \textit{circuit diameter} \cite{bfh-14,bdf-16,kps-17}: the minimum number of circuit steps required to walk from one vertex to another in a polyhedron.

One example of a circuit augmentation algorithm is the \textit{greedy} or \textit{deepest-descent} augmentation scheme from \cite{dhl-15,gdl-15} which requires at most (weakly) polynomially many steps. The challenge in actually implementing this scheme lies in the computation of the required circuits -- a task presumed to be hard. A promising alternative is the \textit{steepest-descent} augmentation scheme of \cite{dhl-15} which generalizes the minimum-mean cycle canceling algorithm for solving network flow problems to any bounded linear program \cite{gdl-15}. The number of steps needed for this scheme is bounded by the number of circuits \cite{dhl-15}. However, a polyhedral model which encodes circuits as vertices can be used to efficiently compute each required circuit \cite{bv-18}. It follows that the algorithm terminates in strongly polynomial time for a polyhedron defined by a totally unimodular matrix \cite{bv-18}.

A limitation of this steepest-descent augmentation scheme is that a separate linear program is required at each iteration to provide the circuits. In this paper, we provide an improved implementation in which each iteration requires only a simple change in variable bounds for a single, dynamic linear program. Therefore, the same program remains in memory throughout the algorithm and each of the computed steepest-descent directions serves as a warm-start for the program in the subsequent iteration. Further, we generalize the scheme so that any steepest-descent direction can be used as an augmenting direction -- even if it is not necessarily a circuit of the underlying polyhedron. This enables modified implementations in which interior point methods are used to solve the dynamic program. 

We begin in \Cref{sec:steepest} by formally defining this generalized steepest-descent augmentation scheme and showing that it satisfies the same favorable properties as the original (\Cref{thm:steepest-descent-augmentation}). Next, in \Cref{sec:implementation} we detail our proposed implementation of the scheme. Lastly, computational results are presented in \Cref{sec:computations}. We compare various versions of our implementation to each other and to a na\"{i}ve approach (\Cref{sec:implementations_comparison}) and discuss the next steps required for adopting steepest-descent augmentation in a viable algorithm for large-scale computations (\Cref{sec:viability_discussion}). 

\section{Steepest-descent Augmentation}\label{sec:steepest}

We outline the basic principles behind a steepest-descent augmentation scheme for linear programs. Let $P = \{ \vex \in \R^n \colon A \vex = \veb, \  B \vex \leq \ved \}$ be a general polyhedron and consider the linear program $LP = \min\{\vecc^T \vex \colon \vex \in P\}$. At iteration $i$ of an augmentation scheme for solving $LP$, assume we have a feasible solution $\vex_i \in P$. If $\vex_i$ is not optimal, the iteration requires an \textit{augmenting direction} $\vey_i$ such that $\vey_i$ is an \textit{improving direction} with respect to $\vecc$ (i.e., $\vecc^T \vey_i < 0$) and such that $\vey_i$ is \textit{strictly feasible} at $\vex_i$ (i.e., there exists some $\alpha > 0$ such that $\vex_i + \alpha \vey_i \in P$). A \textit{maximal step} is then taken in the direction of $\vey_i$ starting at $\vex_i$: That is, $\vex_{i+1} = \vex_i + \alpha_i \vey_i$ where $\vex_i + \alpha_i \vey_i \in P$ but $\vex_i + \alpha_i \vey_i \notin P$ for any $\alpha > \alpha_i$.

In the steepest-descent circuit augmentation scheme from \cite{bv-18,dhl-15}, each of these augmenting directions $\vey_i$ is a so-called \textit{steepest-descent circuit} of $P$ with respect to $\vecc$ at $\vex_i$:

\begin{definition}[Steepest-descent Circuit]\label{def:steepest}
Given a polyhedron $P = \{ \vex \in \R^n \colon A \vex = \veb, \  B \vex \leq \ved \}$, a feasible solution $\vex_0 \in P$, and an objective $\vecc \in \R^n$, a \textit{steepest-descent circuit} at $\vex_0$ is a strictly feasible circuit $\veg \in \mathcal{C}(A,B)$ that minimizes $\vecc^T \veg / ||B \veg||_1$ over all such circuits.
\end{definition}
\noindent It is shown in \cite{bv-18} that such a circuit can be computed by finding a vertex solution to the following LP over a polyhedral model of the set of circuits of $P$:

\begin{align}
\tag{model}
\begin{split}\label{equation:steepest}
\min \ &  \vecc^T \vex \\
\text{s.t.} \  & A\vex = \ve0 \\
& B \vex = \vey^+ - \vey^- \\
& \vey^+_i = 0 \  \ \ \ \   \forall i \colon (B\vex_0)_i = \ved_i\\
 & \sum_{i=1}^{m_B} \vey^+_i + \sum_{i=1}^{m_B}\vey^-_i = 1 \\
  & \vey^+, \vey^- \geq \ve0.
\end{split}
\end{align}

Using this augmentation scheme, no circuit is ever repeated as an augmenting direction and the algorithm terminates in at most $|\mathcal{C}(A,B)|$ iterations \cite{bv-18,dhl-15}. More specifically, the number of iterations is bounded by the number of different values of $\vecc^T \veg / ||B \veg||_1$ over all circuits $\veg \in \mathcal{C}(A,B)$ times the dimension of $P$ -- a result known as \textit{Bland's Theorem} \cite{bj-92,dhl-15}.

When computing an augmenting direction via LP (\ref{equation:steepest}), a vertex solution corresponds to a circuit of $P$ \cite{bv-18}. However, even if the program returns a non-vertex optimal solution (for instance, via an interior point method), we show in \Cref{thm:steepest-descent-augmentation} that the corresponding direction can be used in a generalized steepest-descent augmentation scheme for which the above bounds on the number of iterations still hold. Namely, we define a \textit{steepest-descent augmenting direction} at $\vex_0 \in P$ to be any $\vey \in \ker(A)$ which is strictly feasible at $\vex_0$ and minimizes $\vecc^T \vey / ||B \vey||_1$ over all such directions. A \textit{steepest-descent augmentation} at $\vex_i \in P$ is thus a maximal augmenting step $\vex_{i+1} = \vex_i + \alpha_i \vey_i$ along a steepest-descent augmenting direction $\vey_i$.

\begin{theorem}\label{thm:steepest-descent-augmentation}
Consider the linear program $LP = \min\{ \vecc^T \vex \colon \vex \in P\}$ where $P = \{ \vex \in \R^n \colon A \vex = \veb, B \vex \leq \ved\}$. Given a feasible solution $\vex_0 \in P$, the number of steepest-descent augmentations needed to solve $LP$ is at most
\begin{align*}
    \dim(P) \cdot \left| \left\lbrace\frac{\vecc^T \veg}{ ||B \veg||_1} \colon \veg \in \mathcal{C}(A, B) \right \rbrace \right|.
\end{align*}
\end{theorem}

\begin{proof}
The claim follows from the \textit{conformal sum} property of circuits  \cite{g-75}, which states that any $\vey \in \ker(A)$ can be expressed as the sum $\vey = \sum_{j=1}^t \lambda_j \veg_j$ of at most $\dim(P)$ circuits $\veg_j$ of $P$ (i.e., $t \leq n - m_A$), where $\lambda \geq 0$ and each $B \veg_j$ belongs to the same orthant of $\R^{m_B}$ as $B \vey$. Therefore, any augmenting direction $\vey$ decomposes into at most $\dim(P)$ \textit{conformal circuits} $\veg_1,...,\veg_t$, where each $\veg_j$ is \textit{sign-compatible} with $\vey$ with respect to the matrix $B$. 

Let $\vex_0,\vex_1,...,\vex_k$ be a sequence of steepest-descent augmentations beginning at $\vex_0$. We show that for each augmentation $\vex_{i+1} = \vex_i + \alpha_i \vey_i$, at least one of the conformal circuits of $\vey_i$ must not appear as a conformal circuit in any of the other augmenting directions.

By the definition of steepest-descent augmentations, we retain two useful properties of the (circuit) augmentations from \cite{bv-18,dhl-15}: the steepness of consecutive augmenting directions is non-increasing (i.e., $-\vecc^T \vey_{i+1} / ||B \vey_{i+1}||_1 \leq -\vecc^T \vey_{i} / ||B \vey_{i}||_1$) and a change in orthants from $B\vey_i$ to $B\vey_{i+1}$ implies a strict change in the steepness of the steps. 

Consider the augmenting direction $\vey_i$. Assume $\vey_i$ is not a circuit and let $\veg_1,...,\veg_t$ denote its conformal circuits (i.e., $\vey_i = \sum_{j=1}^t \lambda_j \veg_j$). We claim that the steepness $\vecc^T \vey_{i} / ||B \vey_{i}||_1$ of $\vey_i$ is equal to the steepness $\vecc^T \veg_{j} / ||B \veg_{j}||_1$  for each of the $\veg_j$'s. To see this, note by the definition of conformal circuits that each $\veg_j$ is strictly feasible at $\vex_i$. Hence, by choice of $\vey_i$, none of the $\veg_j$'s is steeper than $\vey_i$. On the other hand, suppose one of the $\veg_j$'s is less steep than $\vey_i$; i.e., $\vecc^T \veg_{j} / ||B \veg_{j}||_1 > \vecc^T \vey_{i} / ||B \vey_{i}||_1$. Without loss of generality, assume $\veg_1$ is less steep than $\vey_1$. We then have:
\begin{align*}
     \vecc^T(\vey_i - \lambda_1 \veg_1) &= \vecc^T \vey_i - \lambda_1  \vecc^T \veg_1 \\
     &= ||B \vey_i||_1 \frac{\vecc^T \vey_i}{||B \vey_i||_1} - ||B \veg_1||_1\frac{ \lambda_1  \vecc^T \veg_1}{||B \veg_1||_1} \\
     &< ||B \vey_i||_1 \frac{\vecc^T \vey_i}{||B \vey_i||_1} - ||B \veg_1||_1\frac{ \lambda_1  \vecc^T \vey_i}{||B \vey_i||_1} \\
     &= \frac{(||B \vey_i||_1 - || B (\lambda_1 \veg_1) ||_1) \cdot \vecc^T \vey_i}{||B\vey_i||_1} \\
     &\leq \frac{(||B (\vey_i - \lambda_1 \veg_1) ||_1) \cdot \vecc^T \vey_i}{||B\vey_i||_1},
\end{align*}
where the second inequality follows from the definition of conformal circuits. This implies:
\begin{align*}
    \frac{\vecc^T(\vey_i - \lambda_1 \veg_1)}{||B (\vey_i - \lambda_1 \veg_1) ||_1} < \frac{\vecc^T \vey_i}{||B \vey_i ||_1}.
\end{align*}
However, since $\vey_i - \lambda_1 \veg_1$ must itself be strictly feasible at $\vex_i$, this contradicts the choice of $\vey_i$. Thus, each $\veg_j$ has the same steepness as $\vey_i$.

It follows that each of the circuits $\veg_j$ can only appear as a conformal circuit while the steepness of augmenting directions is $\vecc^T \vey_{i} / ||B \vey_{i}||_1$. For a fixed steepness value, note that all of the applied augmenting directions must be sign-compatible with each other with respect to $B$. Hence, one of the $\veg_j$'s must not appear as a conformal circuit in any subsequent iteration, else the augmenting step would not have been maximal. Further, at most $\dim(P)$ augmentations can be applied for a fixed steepness value. To see this, note that due to sign-compatibility, if a maximal step terminates at a facet $F$ of $P$, subsequent augmentations for the current steepness value may not leave $F$. Therefore, each step belongs to a face of $P$ with strictly smaller dimension than that of the previous step.

Since the steepness of augmenting directions throughout the steepest-descent scheme is non-decreasing, it follows that the total number of iterations is bounded by the dimension of $P$ times the number of different possible steepness values -- the bound stated in the theorem.
\eoproof
\end{proof}

We note the similarity of this steepest-descent augmentation scheme for linear programs to gradient descent for general optimization problems. Gradient descent computes a steepest direction with respect to the Euclidean norm; in the case of linear minimization, this direction is the negative objective projected onto set of strictly feasible directions at the current solution (i.e., the \textit{inner cone} of the current solution). If a solution is on the boundary of the polyhedron, computing this direction is equivalent to the projection of the cost vector onto a polyhedral cone -- a quadratic programming problem. 

In contrast to the gradient descent direction, computing the steepest-descent direction becomes \textit{easier} as the number of facets containing the current solution increases: Each inequality of $B \vex_0 \leq \ved$ which is tight corresponds to a variable which becomes fixed in LP (\ref{equation:steepest}). Thus, steepest-descent augmentation can be interpreted as a viable implementation of a \textit{maximal-step} gradient descent scheme for linear programs -- instead of using the steepest feasible direction with respect to the Euclidean norm (implicitly assuming that the underlying solution space is a ball in $\R^n$), we compute the steepest feasible direction $\vey$ with respect to the 1-norm of $B\vey$, which takes into account the combinatorial structure of the underlying polyhedron.

\section{Implementation}\label{sec:implementation}

When implementing the steepest-descent augmentation scheme described in \Cref{sec:steepest}, the constraints of LP (\ref{equation:steepest}) change at each iteration according to the active facets at the current feasible solution. In this section, we modify the program so that only the variable upper bounds change at each iteration. Thus, a dynamic model can remain in memory throughout the algorithm and each of the computed steepest-descent directions can be used to warm-start the model in subsequent iterations.

Note that in LP (\ref{equation:steepest}), due to the non-negativity of $\vey^+, \vey^-$ and the constraint $\sum_{i=1}^{m_B} \vey^+_i + \sum_{i=1}^{m_B} \vey^-_i = 1$, each of the variables $\vey^+_i$ and $\vey^-_i$ is implicitly bounded above by $1$. When the current feasible solution $\vex_0$ satisfies $(B \vex_0)_i = \ved_i$, the upper bound for $\vey^+_i$ is strengthened to $0$. Therefore, we can reformulate LP (\ref{equation:steepest}) as follows:

\begin{align}
\tag{steepest}
\begin{split}\label{equation:model}
& \min \   \vecc^T \vex \\
\text{s.t.} \  & A \vex = \ve0 \\
& B \vex - \vey^+ + \vey^- = \ve0 \\
& \vey^+_i \leq 0 \  \ \ \ \   \forall i \colon (B \vex_0)_i = \ved_i\\
& \vey^+_i \leq 1 \  \ \ \ \   \forall i \colon (B \vex_0)_i < \ved_i\\
 & \sum_{i=1}^{m_B}\vey^+_i + \sum_{i=1}^{m_B} \vey^-_i = 1 \\
  & \vey^+, \vey^- \geq \ve0.
\end{split}
\end{align}

Since only the right-hand side of LP (\ref{equation:model}) changes at each iteration, the augmenting direction used in iteration $i$ of the steepest-descent augmentation scheme corresponds to a dual feasible solution to the program at iteration $i + 1$. Hence, warm-starting a dual simplex method to solve LP (\ref{equation:model}) at each iteration is a natural approach for computing the required directions. An outline of this proposed implementation of the steepest-descent augmentation scheme is detailed in \Cref{alg:steepest}.

\begin{algorithm}
\caption{Steepest-descent Augmentation Scheme}\label{alg:steepest}
\begin{algorithmic}[1]
\Procedure{SteepestDescent}{$P, \vecc, \vex_0 \in P$}\Comment{Solves $LP = \min_{\vex \in P} \vecc^T \vex$}
\State Initialize LP (\ref{equation:model}) using $P$, $\vecc$, and $\vex_0$.
\State $i \gets 0$
\State Solve LP (\ref{equation:model}) to obtain optimal solution $(\vex^*, \vey^+, \vey^-)$
\State $\vey_i \gets \vex^*$
\While{$\vecc^T \vey_i < 0$}
\State $\alpha_i \gets \max\{\alpha \in \R^+ \colon \vex_i + \alpha \vey_i \in P\}$
\State $\vex_{i+1} \gets \vex_i + \alpha_i \vey_i$
\State Modify variable upper bounds for LP (\ref{equation:model}) based on $\vex_{i+1}$
\State Using $\vey_{i}$ as a warm-start, re-solve LP (\ref{equation:model}) to obtain optimal solution $(\vex^*, \vey^+, \vey^-)$
\State $\vey_{i+1} \gets \vex^*$
\State $i \gets i + 1$
\EndWhile
\State \textbf{return} $\vex_i$
\EndProcedure
\end{algorithmic}
\end{algorithm}

The success of \Cref{alg:steepest} is dependent upon LP (\ref{equation:model}) being easier to solve than the original LP. Note that although LP (\ref{equation:model}) is formulated within a higher-dimensional space than that of the original problem due to the splitting of $B\vex$ into positive and negative parts, the actual increase in dimension of the feasible set is at most $m_B$ due to the introduced equality constraints. Thus, we can compare this formulation to the introduction of slack variables required when solving the original problem via the simplex method. Additionally, unlike the original problem, LP (\ref{equation:model}) does not depend on the right-hand side vectors $\ved$ or $\veb$. 

Recall also that each active facet at the current solution further reduces the dimension of LP (\ref{equation:model}). Hence, whereas highly degenerate solutions cause inefficiencies in the simplex method, these points make the computation of a steepest-descent circuit easier. Lastly, note that if $P$ has a pair of parallel facets, then the corresponding constraints in LP (\ref{equation:model}) are redundant. In particular, if $B_j = -B_k$, the program is equivalent to the following:
\begin{align*}
\min \ &  \vecc^T \vex \\
\text{s.t.} \  & A\vex = \ve0 \\
& (B \vex)_i = \vey^+_i - \vey^-_i \ \ \ \ \ \ \forall i \neq k\\
& \vey^+_i = 0 \ \ \ \ \ \ \ \ \ \ \ \ \ \ \ \  \ \ \ \   \forall i \neq k \colon (B\vex_0)_i = \ved_i \\
& \vey^-_j = 0 \ \ \ \ \ \ \ \ \ \ \ \ \ \ \ \ \  \ \ \ \text{if } (B \vex_0)_k = \ved_k \\
 & 2\vey^+_j + 2\vey^-_j + \sum_{i \neq j,k}(\vey^+_i  + \vey^-_i) = 1\\
  & \vey^+, \vey^- \geq \ve0.
\end{align*}

Taking all of the above into account, an interesting direction of research would be to determine families of LPs for which solving LP (\ref{equation:model}) becomes significantly easier than solving the original problem. However, even without this expert knowledge, we observe from our computational results in \Cref{sec:viability_discussion} that solving LP (\ref{equation:model}) from scratch is easier than solving the (warm-started) original problem approximately half of the time. When a previous steepest-descent direction is used as a warm-start for LP (\ref{equation:model}), its computation time is faster by an order of magnitude.

Thus, an order-of-magnitude advantage for LP (\ref{equation:model}) is achieved in all but the first iteration of \Cref{alg:steepest}. There are several ways a potentially effective warm-start could be computed for LP (\ref{equation:model}) in the first iteration as well: A steep, feasible direction at $\vex_0$ could be used to warm-start the primal simplex method (such as the first edge direction chosen by the simplex method when applied to the original problem), or a steep but not necessarily feasible direction could be used to warm-start the dual simplex method (such as the projection of the cost vector $\vecc$ onto the affine hull of $P$). A study of these possibilities is another natural direction for future research.

\section{Computational Results}\label{sec:computations}

We implemented the steepest-descent augmentation scheme as outlined in \Cref{alg:steepest} using Gurobi \cite{gurobi} to initialize and repeatedly solve LP (\ref{equation:model}) via the dual simplex method. The algorithm was evaluated on a subset of 79 problems from the Netlib LP Test Set \cite{k-04,netlib}, a repository which serves as a benchmark for comparing the performance of linear programming algorithms on real-life examples. Code for our experiments is available at \URL{https://github.com/charles-viss/steepest-descent}.

First, in \Cref{sec:implementations_comparison}, we evaluate our proposed implementation of the steepest-descent augmentation scheme by comparing it to other alternative implementations; i.e., foregoing the warm-starts at each iteration or using the primal simplex method or an interior point method to repeatedly solve LP (\ref{equation:model}) rather than the dual simplex method. Next, in \Cref{sec:viability_discussion}, we compare the performance of the scheme to that of the simplex method on the original problem and discuss the required next steps towards a viable implementation.

\subsection{Comparison of Implementations}\label{sec:implementations_comparison}
The average (mean and median) results of the experiments comparing different implementations of the steepest-descent scheme are given in \Cref{fig:table1}.  We measured the total running time for each variation or the scheme as well as the number of iterations and the time needed to solve LP (\ref{equation:model}) at each iteration. To evaluate the effectiveness of using warm-starts, we record both the average time needed to compute the first steepest-descent direction as well as the average time needed for all steepest-descent computations. All experiments were performed on an Intel i5 8th Gen CPU.

\begin{table}
\begin{center}
 \begin{tabular}{|c || c | c | c| c|} 
 \hline
  \ \ & \ dual simplex \  & \ interior point \ & \ primal simplex \ & \ dual simplex\** \ \\
 \hline \hline
 Mean Total Time (s) & \textbf{3.932} & 9.243 & 17.055 & 18.193 \\ 
 \hline
 Median Total Time (s) & \textbf{0.875} & 3.109 & 1.814 & 4.300 \\
 \hline \hline
 Mean Avg. Step Time (ms) & \textbf{3.78} & 27.38 & 25.37 & 49.30 \\ 
 \hline
 Median Avg. Step Time (ms) & \textbf{2.04} & 18.22 & 5.89 & 20.47 \\
  \hline \hline
 Mean First Step Time (ms) & 30.20 & \textbf{24.20} & 33.08  & 27.50 \\ 
 \hline
 Median First Step Time (ms) & 15.62 & 15.62 & 15.62 & 15.62 \\
  \hline \hline
 Mean SD Iterations & 231.6 & \textbf{171.1} & 234.1 & 231.6 \\ 
 \hline
 Median SD Iterations & 137.0 & \textbf{119.0} & 134.5 & 137.0 \\
 \hline
\end{tabular}
\end{center}
    \caption{Comparison of the running time, average step times, and number of iterations for different implementations of the steepest-descent scheme. (\**) indicates the dual simplex method without using warm-starts at each iteration.}
    \label{fig:table1}
\end{table}

We first note the success of the dual simplex implementation of the steepest-descent scheme compared to the primal simplex and interior point implementations. Both total running time and average step times are drastically reduced via dual simplex due to the effectiveness of warm-starts. By comparing the first step times to the average step times, we observe that whereas warm-starts for the interior point method do not appear to have a beneficial impact on average step time, warm-starting the dual or primal simplex method results in significant improvement. For the dual simplex method, warm-starts reduce the average step direction computation time by an order of magnitude. For the primal simplex method, a reduction in computation time is still apparent but not nearly as dramatic -- especially when comparing the mean average step times. However, both the interior point and primal simplex implementations (which use warm-starts) significantly outperform the dual simplex implementation when warm-starts are not utilized.

Lastly, we note that the interior point implementation -- which is capable of computing augmenting directions that are not necessarily circuits -- results in significantly fewer iterations than the other two implementations. This suggests that further improvements to the steepest-descent scheme could be achieved by integrating the capability to compute non-circuit augmenting directions into the proposed dual simplex implementation.

\subsection{Towards a Viable Implementation}\label{sec:viability_discussion}

Recall that the steepest-descent augmentation scheme requires an initial feasible solution. To evaluate the viability of the steepest-descent scheme, we therefore compare its performance to that of the primal simplex method warm-started with the same initial solution (i.e., Phase II of the primal simplex method on the original problem). This comparison is detailed in \Cref{fig:table2}, which provides the average total running time and number of iterations for both algorithms as well as the first and average step time for the steepest-descent scheme.

\begin{table}
\begin{center}
 \begin{tabular}{ |c || C{1.5cm} |C{1.5cm}| C{2cm} | C{2.2cm} | C{2.2cm} | C{2.2cm} |} 
 \hline
  SD Results  &  Total (s)  & \ N. Iters \ & First Step (Total, ms) &  First Step (Phase I, ms) & First Step (Phase II, ms) & Average Step (ms) \\
 \hline \hline
 Mean & 3.932 & 231.6 & 30.20 & 8.55 &  21.65 & 3.78 \\ 
 \hline
 Median & 0.875 & 137.0 & 15.62 & 4.98 & 10.64  & 2.04\\
 \hline
\end{tabular}

\vspace{.1in}

 \begin{tabular}{|c ||  c| c|} 
 \hline
  \ Simplex Results \ & \ Total (ms) \ & \ N. Iters \ \\
 \hline \hline
 Mean & 30.24 & 452.6  \\ 
 \hline
 Median & 15.61  & 285.0 \\
 \hline
\end{tabular}
\end{center}
    \caption{Breakdown of the total running time, first step time (Phase I and Phase II), and average step time of the steepest-descent (SD) scheme compared to the performance of the simplex method -- warm-started with the same starting feasible solution as the SD scheme -- on 79 problems from the Netlib LP test set.}
    \label{fig:table2}
\end{table}

We observe that although the \textit{total} running time for the steepest-descent scheme is not competitive with that of the simplex method, computing a first steepest-descent direction via LP (\ref{equation:model}) is comparable to solving the original problem, while computing subsequent directions is significantly easier. We note also that when computing the first steepest-descent direction, approximately one-third of the time is spent in Phase I of the dual simplex method. On the other hand, the simplex method when applied to the original problem starts immediately at Phase II. We therefore see eliminating Phase I of this initial steepest-descent computation as an important, future angle of attack for improving the steepest-descent scheme.

Additionally, we note that although the steepest-descent scheme takes much longer to actually terminate, the decrease in objective value achieved through its first few steps is often similar to that of the simplex method. Consider the results in \Cref{subfig:obj_cycle}, which are representative of the behavior of the steepest-descent scheme when applied to many of the Netlib problems. During the first few iterations of the scheme, relatively large steps are taken and the objective value decreases quickly -- at a similar rate to that of the simplex method. However, after this fast initial drop, the steepest-descent scheme requires many additional iterations to finally converge to an optimal solution.

\begin{figure}[t]
    \begin{subfigure}{0.495\textwidth}
        \includegraphics[width=\textwidth]{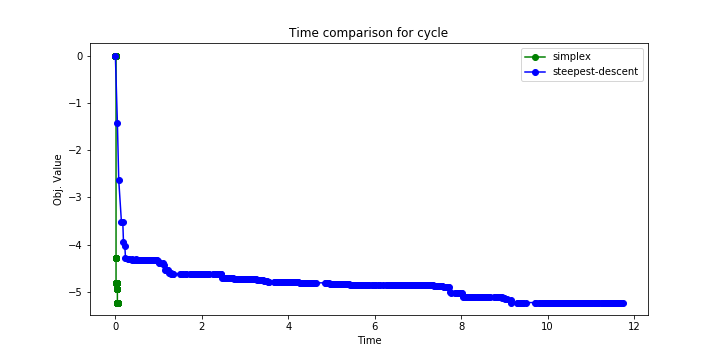}
        \caption{Results for \texttt{cycle}.}\label{subfig:obj_cycle}
    \end{subfigure}
    \begin{subfigure}{0.495\textwidth}
        \includegraphics[width=\textwidth]{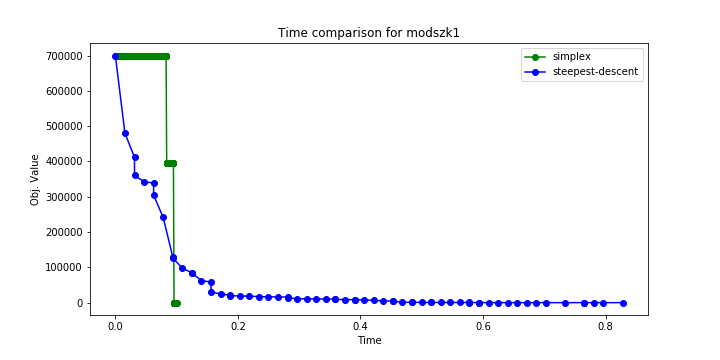}
        \caption{Results for \texttt{modszk1}.}\label{subfig:obj_modszk1}
    \end{subfigure}
    \caption{Plots of the objective value over time for the steepest-descent augmentation scheme and the primal simplex method on two problems from Netlib.}
    \label{fig:results1}
\end{figure}

Another example which exhibits this behavior is given in \Cref{subfig:obj_modszk1}. During its first few iterations, the steepest-descent scheme dramatically reduces the objective value while the simplex method stalls at a degenerate vertex. Thus, an optimal strategy for this problem and for similar problems could be to first use the steepest-descent scheme to quickly improve the objective value via large steps along steep, interior directions. Then, once a certain progress threshold is reached, the algorithm could jump to a nearby vertex and terminate quickly via the simplex method. Using such an approach, an optimal solution could potentially be found faster than the time required by either the steepest-descent scheme or the simplex method individually.

\begin{figure}[t]
    \centering
    \begin{subfigure}{0.325\textwidth}
        \includegraphics[width=\textwidth]{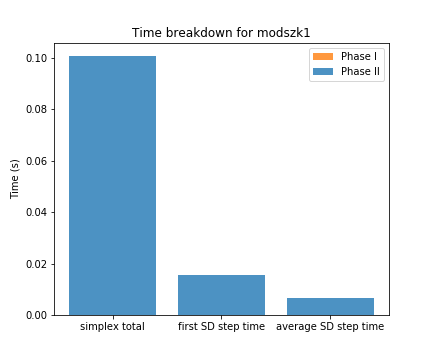}
        \caption{Results for \texttt{modszk1}.}\label{subfig:steps_modszk1}
    \end{subfigure}
    \begin{subfigure}{0.325\textwidth}
        \includegraphics[width=\textwidth]{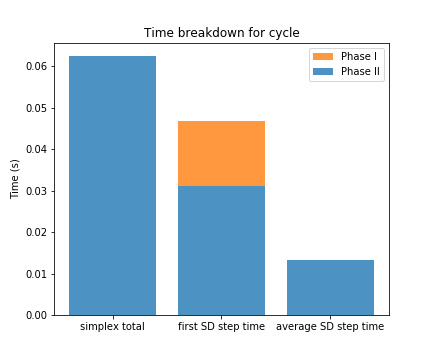}
        \caption{Results for \texttt{cycle}.}\label{subfig:steps_cycle}
    \end{subfigure}
    \begin{subfigure}{0.325\textwidth}
        \includegraphics[width=\textwidth]{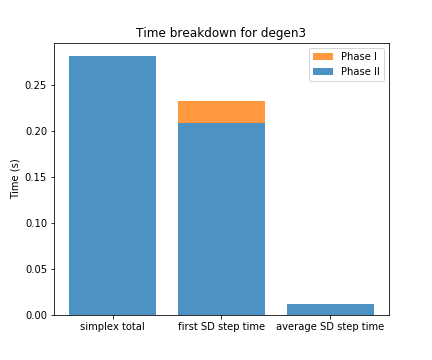}
        \caption{Results for \texttt{degen3}.}\label{subfig:steps_degen3}
    \end{subfigure}
    \caption{Solve time for the simplex method compared to the first and average steepest-descent computation times for three Netlib problems.}  \label{fig:results2}
\end{figure}

Note from \Cref{subfig:steps_modszk1} that one reason the steepest-descent scheme is initially successful in \Cref{subfig:obj_modszk1} is the fast computation time for the first steepest-descent direction. In fact, Phase I of this initial computation is virtually nonexistent. Hence, the problem is structured in such a way that LP (\ref{equation:model}) is relatively easy to solve via the dual simplex method -- even without warm-starts.

For other problems, this reiterates the need for a warm-start for LP~(\ref{equation:model}) at the first iteration. Consider \Cref{subfig:steps_cycle,subfig:steps_degen3}. The time needed to compute the first steepest-descent direction is almost as much as the time needed for the simplex method to solve the original problem. However, in subsequent iterations when a warm-start is available, the time needed to compute steepest-descent directions is significantly reduced. If the first iteration were to be as fast as subsequent iterations, the steepest-descent scheme may again have an opportunity to outperform the simplex method during its initial iterations.

In summary, further steps are needed for a practically viable implementation of the steepest-descent augmentation scheme. Most importantly, as discussed at the end of \Cref{sec:implementation} and as demonstrated by our computational results, a method is needed to quickly compute an effective warm-start for LP (\ref{equation:model}) in the first iteration. Such a direction could be determined based on the underlying LP or computed through tailored algorithms for special families of LPs. There are additional ways in which expert knowledge on the underlying LP could be beneficial: For example, in polyhedra defined by totally unimodular matrices, steps along steepest-descent circuit directions only visit integral solutions \cite{bv-17}. This ensures a lower bound on the decrease in objective value achieved at each iteration. Additionally, it significantly reduces the complexity of computing step sizes and ensures numerical stability -- addressing two issues which we observed for several of the Netlib test problems.

\bibliography{literature}
\bibliographystyle{plain}

\end{document}